\documentclass[10pt]{amsart}
\usepackage[backend=bibtex, style=alphabetic]{biblatex} 
\usepackage{tikz}
\usetikzlibrary{angles}
\usepackage{amssymb}
\usepackage{amsthm}
\usepackage[all]{xy}
\usepackage{microtype}
\usepackage{xcolor}
\usepackage{hyperref}
\usepackage[nameinlink]{cleveref}
\usepackage{graphicx}
\usepackage{comment}
\usepackage{bbm}
\usepackage{algpseudocode}
\usepackage{algorithm}

\newcommand{\midmid}{\mathrel{}\middle|\mathrel{}}

\makeatletter
    
    \@addtoreset{equation}{section}
  \makeatother

\hypersetup{
 colorlinks,
 linkcolor={teal},
 citecolor={teal},
 urlcolor={teal}
}

\calclayout

\DeclareFieldFormat
  [article,book,inbook,incollection,inproceedings,patent,thesis,unpublished]
  {title}{\emph{#1\isdot}}

\theoremstyle{plain}
	\newtheorem{theorem}{Theorem}
	\newtheorem{lemma}[theorem]{Lemma}
        
	\newtheorem{corollary}[theorem]{Corollary}
	
	\newtheorem{proposition}[theorem]{Proposition}
	\newtheorem{question}[theorem]{Question}
        \newtheorem{conjecture}[theorem]{Conjecture}

\theoremstyle{definition} 
	\newtheorem{remark}[theorem]{Remark}
	\newtheorem{definition}[theorem]{Definition}
	\newtheorem{example}[theorem]{Example}

\begin{document}
\title{Irrationality of the reciprocal sum of doubly exponential sequences}
\author[J. Koizumi]{Junnosuke Koizumi}
\address{RIKEN iTHEMS, Wako, Saitama 351-0198, Japan}
\email{junnosuke.koizumi@riken.jp}

\date{\today}
\thanks{}

\begin{abstract}
We show that sequences of positive integers whose ratios \(a_n^2/a_{n+1}\) lie within a specific range are almost uniquely determined by their reciprocal sums. For instance, the Sylvester sequence is uniquely characterized as the only sequence with \( a_n^2/a_{n+1}\in [2/3,4/3] \) whose reciprocal sum is equal to \(1\).
This result has applications to irrationality problems. We prove that for almost every real number \(\alpha > 1\), sequences asymptotic to \(\alpha^{2^n}\) have irrational reciprocal sums. Furthermore, our observations provide heuristic insight into an open problem by Erd\H{o}s and Graham.
\end{abstract}

\maketitle
\setcounter{tocdepth}{1}
\tableofcontents

\enlargethispage*{20pt}
\thispagestyle{empty}

\section*{Introduction}

The asymptotic behavior of a sequence of positive integers is related to the irrationality of its reciprocal sum. It is a folklore result that if a sequence satisfies \( \lim_{n\to\infty} a_n^{2^{-n}} = \infty \), then its reciprocal sum is irrational. Therefore, the largest possible asymptotic growth of a sequence for which the rationality of the reciprocal sum can be expected is \( C^{2^n} \) for some constant \(C>1\). A well-known example of such sequence is the (shifted) \emph{Sylvester sequence} \cite[A129871]{OEIS}, which is defined by
$$
    s_1=2,\quad s_{n+1}=s_n^2-s_n+1.
$$
It is straightforward to see that the sum of their reciprocals is $1$:
\[
1=\sum_{n=1}^\infty\dfrac{1}{s_n}=\dfrac{1}{2}+\dfrac{1}{3}+\dfrac{1}{7}+\dfrac{1}{43}+\dfrac{1}{1807}+\cdots.
\]
Also, it is known that there is a constant
\(
c=1.2640847\cdots
\)
such that $s_n\approx c^{2^n}$ (see \cite[p. 109]{Concrete}).
On the other hand, the \emph{Millin series} \cite{Millin} provides an example of a doubly exponential sequence with an irrational reciprocal sum:
$$
    \dfrac{5-\sqrt{5}}{2}=\sum_{n=1}^\infty \dfrac{1}{F_{2^{n}}}=\dfrac{1}{1}+\dfrac{1}{3}+\dfrac{1}{21}+\dfrac{1}{987}+\dfrac{1}{2178309}+\cdots.
$$

The purpose of this paper is to point out that these doubly exponential sequences are often \emph{almost uniquely} determined by their reciprocal sum.
Our main result is the following:

\begin{theorem}\label{thm:main}
    Let $\beta\geq 0$ be a real number, and $(a_n)_{n=1}^\infty$ be a sequence of positive integers satisfying
    \[
    \left|\dfrac{a_n^2}{a_{n+1}}-\beta\right|\leq\dfrac{1}{3},\quad \sum_{n=1}^\infty\dfrac{1}{a_n}=r<\infty.
    \]
    Then, for every $n$ satisfying $a_n\geq 8(\beta+(1/3))^2$, we have
    \[
    a_n = \left\lfloor\left(r-\sum_{k=1}^{n-1}\dfrac{1}{a_k}\right)^{-1}+\beta\right\rceil,
    \]
    where $\lfloor x\rceil=\lfloor x+(1/2)\rfloor$ is the integer closest to $x$.
    Moreover, if $\lim_{n\to \infty}a_n^2/a_{n+1}=\beta$, then we have
    \[
    \lim_{n\to \infty}\left|\left(r-\sum_{k=1}^{n-1}\dfrac{1}{a_k}\right)^{-1}+\beta-a_n\right|=0.
    \]
\end{theorem}
This yields the following characterizations of the Sylvester sequence and the Millin series:

\begin{corollary}\label{cor:Sylvester}
    Let $(a_n)_{n=1}^\infty$ be a sequence of positive integers satisfying
    \[
    \dfrac{2}{3}\leq \dfrac{a_n^2}{a_{n+1}}\leq \dfrac{4}{3},\quad \sum_{n=1}^\infty\dfrac{1}{a_n} = 1.
    \]
    Then, we have $a_n=s_n$, where $(s_n)_{n=1}^\infty$ is the Sylvester sequence.
\end{corollary}

\begin{corollary}\label{cor:Millin}
    Let $(a_n)_{n=1}^\infty$ be a sequence of positive integers satisfying
    \[
    \dfrac{a_n^2}{a_{n+1}}\leq \dfrac{2}{3},\quad \sum_{n=1}^\infty\dfrac{1}{a_n} = \dfrac{5-\sqrt{5}}{2}.
    \]
    Then, we have $a_n=F_{2^n}$, where $(F_n)_{n=1}^\infty$ is the Fibonacci sequence.
\end{corollary}

Note that \Cref{cor:Sylvester} resembles Badea's characterization of the Sylvester-like sequences \cite{Badea}: if $(a_n)_{n=1}^\infty$ is a sequence of positive integers satisfying $a_{n+1}\geq a_n^2-a_n+1$ for $n\gg 0$ and $\sum_{n=1}^\infty(1/a_n)\in \mathbb{Q}$, then the equality $a_{n+1}= a_n^2-a_n+1$ holds for $n\gg 0$.

\Cref{thm:main} has an immediate application to the study of \emph{irrationality sequences} defined by Erd\H{o}s-Straus \cite{Erdos75} and Erd\H{o}s-Graham \cite{Erdos_Graham_80}.
Following the terminology introduced by Kova\v{c}-Tao \cite{Kovac_Tao}, we say that an increasing sequence\footnote{In the original definition \cite{Erdos_Graham_80}, the sequence was required to be strictly increasing, i.e., $a_1<a_2<a_3<\cdots$. In this paper, we relax this condition by allowing equality, enabling us to consider sequences such as $\lfloor \alpha^{2^n}\rfloor$.} of positive integers
$$
    a_1\leq a_2\leq a_3\leq \cdots
$$
is a \emph{Type 2 irrationality sequence} if for every sequence of positive integers $(b_n)_{n=1}^\infty$ such that $a_n\approx b_n$, we have $\sum_{n=1}^\infty(1/b_n)\not\in\mathbb{Q}$.
For example, any sequence with $a_n^{2^{-n}}\to \infty$ is a Type 2 irrationality sequence by the folklore result mentioned above.
On the other hand, Kova\v{c}-Tao \cite{Kovac_Tao} proved that any sequence satisfying $a_n^2/a_{n+1}\to \infty$ (e.g. $\lfloor 2^{(2-\varepsilon)^n}\rfloor$ for $0<\varepsilon<1$) cannot be a Type 2 irrationality sequence.
An unsolved problem of Erd\H{o}s-Graham \cite[p. 63]{Erdos_Graham_80} asks whether $2^{2^n}$ is a Type 2 irrationality sequence; it is also listed in the website \emph{Erd\H{o}s Problems} \cite[Problem \#263]{EP}.
As a consequence of \Cref{thm:main}, we get the following result:

\begin{theorem}\label{type2_countable}
    Let $\mathcal{I}$ denote the following subset of $(1,\infty)$:
    $$
    \mathcal{I}:=\{\alpha\in (1,\infty)\mid \lfloor\alpha^{2^n}\rfloor\text{ is a Type 2 irrationality sequence}\}.
    $$
    Then, its complement $(1,\infty)\setminus \mathcal{I}$ is countable.
\end{theorem}

In the latter half of the paper, we study the following unsolved problem of Erd\H{o}s-Graham:

\begin{question}[{Erd\H{o}s-Graham \cite[p. 64]{Erdos_Graham_80}}]\label{ques:erdos}
    Let $(a_n)_{n=1}^\infty$ be a sequence of positive integers satisfying
    $$
        \lim_{n\to \infty}\dfrac{a_n^2}{a_{n+1}}=1\quad\text{and}\quad
        \sum_{n=1}^\infty \dfrac{1}{a_n}\in \mathbb{Q}.
    $$
    Is it true that $a_{n+1}=a_n^2-a_n+1$ holds for $n\gg 0$?
\end{question}

In order to state our result, we introduce the concept of \emph{pseudo-greedy expansion}, which is a variant of the greedy expansion of a positive real number into unit fractions.
For a positive real number $r$, its pseudo-greedy expansion is the sequence of positive integers $(a_n)_{n=1}^\infty$ defined by
$$
    a_n=\left\lfloor\left( r-\sum_{k=1}^{n-1}\dfrac{1}{a_k} \right)^{-1}+1\right\rceil,
$$
where $\lfloor x\rceil=\lfloor x+(1/2)\rfloor$ is the integer closest to $x$.
We can show that $r=\sum_{n=1}^\infty(1/a_n)$, so it indeed gives an expansion of $r$ into unit fractions.
We define the \emph{gap sequence} of this expansion by
$$
    \varepsilon_n=\left( r-\sum_{k=1}^{n-1}\dfrac{1}{a_k} \right)^{-1}+1-a_n.
$$
Then, we prove that \Cref{ques:erdos} is equivalent to the following conjecture:

\begin{conjecture}\label{conj}
    Let $r$ be a positive rational number and $(\varepsilon_n)_{n=1}^\infty$ be the gap sequence of the pseudo-greedy expansion of $r$.
    If $\lim_{n\to \infty}\varepsilon_n= 0$, then $\varepsilon_n=0$ holds for $n\gg 0$.
\end{conjecture}

We expect that \Cref{conj} is true even without assuming $\lim_{n\to \infty}\varepsilon_n=0$.
We confirmed by a computer that \Cref{conj} holds for $r=p/q$ with $0<p\leq q\leq 10^5$.
In \Cref{rem:heuristic}, we provide a heuristic argument showing that \Cref{conj} is likely to be correct.

\subsection*{Notation}

For two sequences of real numbers $(x_n)_{n=1}^\infty$, $(y_n)_{n=1}^\infty$ with $y_n>0$, we write:
\begin{itemize}
    \item $x_n=O(y_n)$ if $\limsup_{n\to \infty}(|x_n|/y_n)<\infty$;
    \item $x_n=o(y_n)$ if $\lim_{n\to \infty}(x_n/y_n)=0$;
    \item $x_n\approx y_n$ if $\lim_{n\to \infty}(x_n/y_n)=1$.
\end{itemize}
We say that a proposition $P(n)$ is true for $n\gg 0$ if there exists a positive integer $n_0$ such that $P(n)$ is true for all $n\geq n_0$.

\subsection*{Acknowledgments}
The author would like to thank Takahiro Ueoro for carefully reading the paper and for informing him about the problem of odd greedy expansion.
The author also thanks Vjekoslav Kova\v{c} for providing comments on an earlier version of this paper.

\section{Proof of the main results}

The idea of the proof of \Cref{thm:main} is very simple.
Let $(a_n)_{n=1}^\infty$ be a sequence of positive integers such that $a_n^2/a_{n+1}$ is sufficiently close to $\beta\geq 0$ and $a_n$ is sufficiently large.
Then, we have
\begin{align}\label{eq:remainder}
r-\sum_{k=1}^{n-1}\dfrac{1}{a_k} = \dfrac{1}{a_n}+\dfrac{1}{a_{n+1}}+\cdots =\dfrac{1}{a_n}\left(1+\sum_{k=1}^\infty\dfrac{a_n}{a_{n+k}}\right).
\end{align}
Using that the sum in the right hand side is very small, we get
\[
\left(r-\sum_{k=1}^{n-1}\dfrac{1}{a_k}\right)^{-1}\fallingdotseq
a_n\left(1-\sum_{k=1}^\infty\dfrac{a_n}{a_{n+k}}\right)
=a_n-\dfrac{a_n^2}{a_{n+1}}-\sum_{k=2}^\infty\dfrac{a_n^2}{a_{n+k}}.
\]
Since the last sum is also very small and $a_n^2/a_{n+1}$ is close to $\beta$, it follows that the left hand side is close to $a_n-\beta$, which is what we want.
We will make this precise:

\begin{proof}[Proof of \Cref{thm:main}]
    Let $\beta\geq 0$ be a real number, and $(a_n)_{n=1}^\infty$ be a sequence of positive integers such that
    \[
    \left|\dfrac{a_n^2}{a_{n+1}}-\beta\right|\leq\dfrac{1}{3},\quad \sum_{n=1}^\infty\dfrac{1}{a_n}=r<\infty.
    \]
    We write
    \(\beta_+ = \beta + (1/3)\).
    Fix a positive integer $n$ satisfying $a_n\geq 8\beta_+^2$.
    Combining with $a_n\geq 1$, we obtain
    \( a_n^2\geq 8\beta_+^2 \)
    and hence
    \(a_n\geq 2\sqrt{2}\beta_+\).
    In other words, we have
    \[
    \dfrac{\beta_+^2}{a_n}\leq \dfrac{1}{8},\quad \dfrac{\beta_+}{a_n}\leq\dfrac{1}{2\sqrt{2}}.
    \]
    Our assumption shows that $a_{k+1}\geq a_k^2/\beta_+$ holds for all $k>0$.
    Using this inequality iteratively, we obtain
    \[
    a_{n+k}\geq \dfrac{a_{n+k-1}^2}{\beta_+}\geq\dfrac{a_{n+k-2}^4}{\beta_+^3}\geq\cdots\geq\dfrac{a_n^{2^k}}{\beta_+^{2^k-1}}=\beta_+\left(\dfrac{a_n}{\beta_+}\right)^{2^k}
    \]
    for all $k\geq 0$.
    Therefore, we have
    \begin{align}\label{eq:estimate_A}
        A:=\sum_{k=1}^\infty\dfrac{a_n}{a_{n+k}}&{}\leq\dfrac{a_n}{\beta_+}\sum_{k=1}^\infty \left(\dfrac{\beta_+}{a_n}\right)^{2^k}\leq\dfrac{a_n}{\beta_+}\sum_{k=1}^\infty\left(\dfrac{\beta_+}{a_n}\right)^{2k}\\
        &{}=\dfrac{\beta_+}{a_n}\cdot\dfrac{1}{1-(\beta_+/a_n)^2}\leq\dfrac{\beta_+}{a_n}\cdot\dfrac{1}{1-(1/8)}=\dfrac{8}{7}\cdot\dfrac{\beta_+}{a_n}.\notag
    \end{align}
    Similarly, we have
    \begin{align}\label{eq:estimate_B}
        B:=\sum_{k=2}^\infty\dfrac{a_n^2}{a_{n+k}}&{}\leq\dfrac{a_n^2}{\beta_+}\sum_{k=2}^\infty\left(\dfrac{\beta_+}{a_n}\right)^{2^k}
        \leq\dfrac{a_n^2}{\beta_+}\sum_{k=1}^\infty\left(\dfrac{\beta_+}{a_n}\right)^{4k}\\
        &{}= \dfrac{\beta_+^2}{a_n}\cdot\dfrac{\beta_+}{a_n}\cdot\dfrac{1}{1-(\beta_+/a_n)^4}\leq \dfrac{2\sqrt{2}}{63}.\notag
    \end{align}
    Now we use \eqref{eq:remainder}.
    Taking the reciprocal, we get
    \begin{align}\label{eq:reciprocal_expansion}
    \left(r-\sum_{k=1}^{n-1}\dfrac{1}{a_k}\right)^{-1}
    &{}=\dfrac{a_n}{1+A}=a_n-a_nA+\dfrac{a_nA^2}{1+A}\\
    &{}=a_n-\dfrac{a_n^2}{a_{n+1}}-B+\dfrac{a_nA^2}{1+A}.\notag
    \end{align}
    By \eqref{eq:estimate_B}, we have $B<(1/6)$.
    On the other hand, by \eqref{eq:estimate_A}, we have
    \[
        \dfrac{a_nA^2}{1+A}\leq a_nA^2\leq \dfrac{64}{49}\cdot\dfrac{\beta_+^2}{a_n}\leq \dfrac{8}{49}<\dfrac{1}{6}.
    \]
    Therefore, we have
    \[
    \left|\left(r-\sum_{k=1}^{n-1}\dfrac{1}{a_k}\right)^{-1}+\dfrac{a_n^2}{a_{n+1}}-a_n\right|=\left|-B+\dfrac{a_nA^2}{1+A}\right|<\dfrac{1}{6}.
    \]
    Finally, by our assumption that $|(a_n^2/a_{n+1})-\beta|\leq (1/3)$, we obtain
    \[
    \left|\left(r-\sum_{k=1}^{n-1}\dfrac{1}{a_k}\right)^{-1}+\beta-a_n\right|<\dfrac{1}{6}+\dfrac{1}{3}=\dfrac{1}{2}.
    \]
    This shows that $a_n$ is the integer closest to $(r-\sum_{k=1}^{n-1}(1/a_k))^{-1}+\beta$.

    Suppose, moreover, that $\lim_{n\to \infty}(a_n^2/a_{n+1})=\beta$.
    The inequalities \eqref{eq:estimate_A} and \eqref{eq:estimate_B} shows that $a_nA^2$ and $B$ converge to $0$ as $n\to \infty$.
    Therefore, by \eqref{eq:reciprocal_expansion}, we have
    \[
    \left(r-\sum_{k=1}^{n-1}\dfrac{1}{a_k}\right)^{-1} = a_n - \dfrac{a_n^2}{a_{n+1}} + o(1) = a_n - \beta + o(1).
    \]
    In other words, the difference $|(r-\sum_{k=1}^{n-1}(1/a_k))^{-1}+\beta-a_n|$ converges to $0$ as $n\to \infty$.
\end{proof}

\begin{proof}[Proof of \Cref{cor:Sylvester}]
    First, we note that the Sylvester sequence satisfies
    \[
    \sum_{k=1}^{n-1}\dfrac{1}{s_k} = 1-\dfrac{1}{s_n-1},
    \]
    by the recurrence relation $s_{n+1}=s_n^2-s_n+1$.
    Therefore, we also have
    \begin{align}\label{eq:sylvester_recurrence}
    s_n=\left(1-\sum_{k=1}^{n-1}\dfrac{1}{s_k}\right)^{-1}+1.
    \end{align}
    Let $(a_n)_{n=1}^\infty$ be another sequence of positive integers satisfying
    \[
    \dfrac{2}{3}\leq \dfrac{a_n^2}{a_{n+1}}\leq \dfrac{4}{3},\quad \sum_{n=1}^\infty\dfrac{1}{a_n} = 1.
    \]
    Applying \Cref{thm:main} with $\beta=1$, we obtain
    \begin{align}\label{eq:sylvester_recurrence_2}
    a_n\geq 15\implies a_n=\left\lfloor\left(1-\sum_{k=1}^{n-1}\dfrac{1}{a_k}\right)^{-1}+1\right\rceil.
    \end{align}
    Since $\sum_{n=1}^\infty(1/a_n)=1$, we have $a_1\geq 2$.
    Using the inequality $a_{n+1}\geq (3/4)a_n^2$ inductively, we obtain
    \[
    a_1\geq 2,\quad a_2\geq 3,\quad a_3\geq 7,\quad a_4\geq 37.
    \]
    In particular, \eqref{eq:sylvester_recurrence_2} implies
    \[
    37\leq a_4=\left\lfloor\left(1-\sum_{k=1}^{3}\dfrac{1}{a_k}\right)^{-1}+1\right\rceil,
    \]
    and hence
    \[
    \sum_{k=1}^3\dfrac{1}{a_k}\geq \dfrac{69}{71} >\dfrac{1}{2}+\dfrac{1}{3} + \dfrac{1}{8}.
    \]
    Therefore, we must have $a_1=2$, $a_2=3$, $a_3=7$, and the rest of the sequence is determined by \eqref{eq:sylvester_recurrence_2}.
    Comparing this with \eqref{eq:sylvester_recurrence}, we conclude that $(a_n)_{n=1}^\infty$ coincides with the Sylvester sequence.
\end{proof}

\begin{proof}[Proof of \Cref{cor:Millin}]
    Let $\phi=(1+\sqrt{5})/2$ and $\overline{\phi}=(1-\sqrt{5})/2$.
    Then, the Fibonacci sequence can be written as
    \(
    F_n=(\phi^n-\overline{\phi}^n)/\sqrt{5}
    \),
    and hence we have
    \[
    \dfrac{F_{2^n}^2}{F_{2^{n+1}}} = \dfrac{1}{\sqrt{5}}\cdot\dfrac{\phi^{2^{n+1}}-2+\overline{\phi}^{2^{n+1}}}{\phi^{2^{n+1}}-\overline{\phi}^{2^{n+1}}}\leq \dfrac{1}{\sqrt{5}}\leq \dfrac{2}{3}.
    \]
    Applying \Cref{thm:main} with $\beta=1/3$, we obtain
    \begin{align}\label{eq:Millin_recurrence}
    n\geq 3\implies F_{2^n}=\left\lfloor\left(\dfrac{5-\sqrt{5}}{2}-\sum_{k=1}^{n-1}\dfrac{1}{F_{2^k}}\right)^{-1}+\dfrac{1}{3}\right\rceil.
    \end{align}
    Let $(a_n)_{n=1}^\infty$ be another sequence of positive integers satisfying
    \[
    \dfrac{a_n^2}{a_{n+1}}\leq \dfrac{2}{3},\quad \sum_{n=1}^\infty\dfrac{1}{a_n} = \dfrac{5-\sqrt{5}}{2}.
    \]
    Applying \Cref{thm:main} with $\beta=(1/3)$, we obtain
    \begin{align}\label{eq:Millin_recurrence_2}
    a_n\geq 4\implies a_n=\left\lfloor\left(\dfrac{5-\sqrt{5}}{2}-\sum_{k=1}^{n-1}\dfrac{1}{a_k}\right)^{-1}+\dfrac{1}{3}\right\rceil.
    \end{align}
    Using the inequality $a_{n+1}\geq (3/2)a_n^2$ inductively, we obtain
    \[
    a_1\geq 1,\quad a_2\geq 2,\quad a_3\geq 6,\quad a_4\geq 54.
    \]
    In particular, \eqref{eq:Millin_recurrence_2} implies
    \[
    54\leq a_4 = \left\lfloor\left(\dfrac{5-\sqrt{5}}{2}-\sum_{k=1}^{3}\dfrac{1}{a_k}\right)^{-1}+\dfrac{1}{3}\right\rceil,
    \]
    and hence
    \begin{align}\label{eq:Millin_first_three}
    \sum_{k=1}^{3}\dfrac{1}{a_k}\geq \dfrac{1583 - 319 \sqrt{5}}{638} = 1.3631572\cdots.
    \end{align}
    This easily implies that $a_1=1$ and $a_2=3$.
    The rest of the sequence is determined by \eqref{eq:Millin_recurrence_2}.
    Comparing this with \eqref{eq:Millin_recurrence}, we conclude that $(a_n)_{n=1}^\infty$ coincides with $(F_{2^n})_{n=1}^\infty$.
\end{proof}

\begin{proof}[Proof of \Cref{type2_countable}]
    By definition, $(1,\infty)\setminus\mathcal{I}$ consists of real numbers $\alpha>1$ such that there exists a sequence of positive integers $(a_n)_{n=1}^\infty$ satisfying
    \[
    a_n\approx \alpha^{2^n}\quad\text{and}\quad \sum_{n=1}^\infty\dfrac{1}{a_n} \in \mathbb{Q}.
    \]
    Let $\mathcal{S}$ denote the set of all such sequences, i.e.,
    $$
    \mathcal{S}:=\bigcup_{\alpha>1}\left\{(a_n)_{n=1}^\infty \midmid a_n\approx \alpha^{2^n},\ \sum_{n=1}^\infty\dfrac{1}{a_n}\in \mathbb{Q}\right\}.
    $$
    Then, there is a surjective map
    $$
        \mathcal{S}\to (1,\infty)\setminus\mathcal{I};\quad (a_n)_{n=1}^\infty\mapsto \lim_{n\to \infty} a_n^{2^{-n}}.
    $$
    It suffices to show that $\mathcal{S}$ is countable.
    Let $(a_n)_{n=1}^\infty$ be an element of $\mathcal{S}$.
    Then, we have $\lim_{n\to \infty}(a_n^2/a_{n+1})=1$ and $\lim_{n\to \infty}a_n=\infty$.
    By \Cref{thm:main}, there exists an integer $n_0>0$ such that
    \[
    n\geq n_0\implies a_n = \left\lfloor\left(r-\sum_{k=1}^{n-1}\dfrac{1}{a_k}\right)^{-1}+1\right\rceil,
    \]
    where $r=\sum_{n=1}^\infty(1/a_n)$.
    In particular, the sequence $(a_n)_{n=1}^\infty$ is uniquely determined by the tuple
    $(a_1,\dots,a_{n_0-1},r)$.
    Since the set of all such tuples is countable, it follows that $\mathcal{S}$ is countable.
\end{proof}

\section{Pseudo-greedy expansion}

Given a positive real number $r$, there are many possible ways to express $r$ as a (possibly infinite) sum of unit fractions.
The simplest choice is the \emph{greedy expansion}, which is given by
$$
    r=\dfrac{1}{a_1}+\dfrac{1}{a_2}+\dfrac{1}{a_3}+\cdots,\quad a_n=\left\lceil\left(r-\sum_{k=1}^{n-1}\dfrac{1}{a_k}\right)^{-1}\right\rceil.
$$
It can be easily shown that the greedy expansion of any positive rational number terminates in finite steps.
Another interesting choice is the \emph{odd greedy expansion}, where $a_n$ is defined to be the smallest odd number $\geq (r-\sum_{k=1}^{n-1}(1/a_k))^{-1}$.
It is an open problem whether the odd greedy expansion of a positive rational number with odd denominator terminates in finite steps (see \cite[p. 88]{Guy_unsolved}).

Motivated by \Cref{thm:main}, we study the following variant of the greedy expansion:

\begin{definition}
    Let $r$ be a positive real number.
    Define a sequence of positive integers $(a_n)_{n=1}^\infty$ by
    $$
        a_n=\left\lfloor \left(r-\sum_{k=1}^{n-1}\dfrac{1}{a_k}\right)^{-1}+1\right\rceil.
    $$
    We call $(a_n)_{n=1}^\infty$ the \emph{pseudo-greedy expansion} of $r$.
\end{definition}

\begin{lemma}
    Let $r$ be a positive real number and $(a_n)_{n=1}^\infty$ be its pseudo-greedy expansion.
    Then, we have
    \(r = \sum_{n=1}^{\infty}(1/a_n)\).
\end{lemma}

\begin{proof}
    We define a sequence of positive real numbers $(x_n)_{n=1}^\infty$ by
    $x_n=r-\sum_{k=1}^{n-1}(1/a_k)$.
    Since $(x_n)_{n=1}^\infty$ is decreasing, we have $\lim_{n\to \infty}x_n=\gamma$ for some real number $\gamma\geq 0$.
    By definition, we have $a_n=\lfloor x_n^{-1} +1\rceil\leq x_n^{-1}+2$ and hence
    $$
        x_{n+1}=x_n-\dfrac{1}{a_n}\leq x_n - \dfrac{1}{x_n^{-1}+2} = x_n\left(1-\dfrac{1}{1+2x_n}\right).
    $$
    Taking the limit as $n\to \infty$, we obtain
    $$
        \gamma\leq \gamma\left(1-\dfrac{1}{1+2\gamma}\right)
    $$
    and hence $\gamma=0$.
    Therefore, we have $r = \sum_{n=1}^{\infty}(1/a_n)$.
\end{proof}

\begin{definition}
    Let $r$ be a positive real number and $(a_n)_{n=1}^\infty$ be its pseudo-greedy expansion.
    We define a sequence of positive real numbers $(x_n)_{n=1}^\infty$ by
    $$
        x_n=r-\sum_{k=1}^{n-1}\dfrac{1}{a_k}.
    $$
    We call $(x_n)_{n=1}^\infty$ the \emph{remainder sequence} of the pseudo-greedy expansion of $r$.
    By definition, we have $a_n=\lfloor x_n^{-1}+1\rceil$.
    We also define a sequence of real numbers $(\varepsilon_n)_{n=1}^\infty$ by
    $$
        \varepsilon_n = x_n^{-1}+1-a_n.
    $$
    We call $(\varepsilon_n)_{n=1}^\infty$ the \emph{gap sequence} of the pseudo-greedy expansion of $r$.
\end{definition}

As an immediate consequence of \Cref{thm:main}, we obtain the following:
\begin{corollary}\label{cor:main}
    Let $(a_n)_{n=1}^\infty$ be a sequence of positive integers satisfying
    $$
        \lim_{n\to \infty}\dfrac{a_n^2}{a_{n+1}}=1\quad\text{and}\quad
        \sum_{n=1}^\infty \dfrac{1}{a_n}<\infty.
    $$
    Then, there is some integer $N\geq 0$ such that $(a_{N+n})_{n=1}^\infty$ is the pseudo-greedy expansion of
    \(\sum_{n=1}^\infty(1/a_{N+n})\).
    Moreover, the gap sequence of this expansion satisfies $\lim_{n\to \infty}\varepsilon_n=0$.
\end{corollary}

\begin{example}\label{example:sylvester}
    The pseudo-greedy expansion of $1$ is given by the Sylvester sequence:
    $$
        1=\dfrac{1}{2}+\dfrac{1}{3}+\dfrac{1}{7}+\dfrac{1}{43}+\dfrac{1}{1807}+\cdots.
    $$
    More generally, for any integer $m>0$, the pseudo-greedy expansion of $1/m$ is given by $a_n=s_n(m)$, where
    $$
        s_1(m)=m+1,\quad s_{n+1}(m)=s_n(m)^2-s_n(m)+1.
    $$
    The remainder/gap sequences are given by $x_n = 1/(s_n(m)-1)$, $\varepsilon_n=0$.
    It is known that there is a constant $c(m)>1$ such that
    $s_n(m)\approx c(m)^{2^n}$ (see \cite[p. 109]{Concrete}).
    Wagner-Ziegler \cite{Wagner_Ziegler} showed that $c(m)$ is irrational, and Dubickas \cite{Dubickas} showed that $c(m)$ is transcendental.
\end{example}

The recurrence relation $a_{n+1}=a_n^2-a_n+1$ appearing in \Cref{example:sylvester} is a special case of the following:

\begin{lemma}\label{lem:PG_sylvester}
    Let $r$ be a positive real number, and $(a_n)_{n=1}^\infty$, $(\varepsilon_n)_{n=1}^\infty$ be its pseudo-greedy expansion and the gap sequence.
    Then, we have
    $$
        a_{n+1}=\dfrac{1}{1-\varepsilon_n}a_n^2-a_n+(1-\varepsilon_{n+1}).
    $$
\end{lemma}

\begin{proof}
    By definition, we have $x_n^{-1}=a_n-(1-\varepsilon_n)$ and hence
    \[
    x_{n+1}=x_n-\dfrac{1}{a_n}=\dfrac{1}{a_n-(1-\varepsilon_n)}-\dfrac{1}{a_n}=\dfrac{1-\varepsilon_n}{a_n^2-(1-\varepsilon_n)a_n}.
    \]
    Substituting this into $a_{n+1}=x_{n+1}^{-1}+1-\varepsilon_{n+1}$, we obtain the desired formula.
\end{proof}

\begin{lemma}
    Let $r$ be a positive real number, and $(\varepsilon_n)_{n=1}^\infty$ be the gap sequence of the pseudo-greedy expansion of $r$.
    Then, we have
    $$
        \varepsilon_n=0\implies \varepsilon_{n+1}=0.
    $$
\end{lemma}

\begin{proof}
    Let $(a_n)_{n=1}^\infty$, $(x_n)_{n=1}^\infty$ be the pseudo-greedy expansion of $r$ and its remainder sequence.
    If $\varepsilon_n=0$, then $x_n^{-1}$ is an integer, so we can write
    \(x_n=1/m\)
    for some positive integer $m$.
    Then, we have $a_n = m+1$ and hence
    \[
    x_{n+1}=x_n-\dfrac{1}{a_n} = \dfrac{1}{m(m+1)}.
    \]
    This shows that $x_{n+1}^{-1}$ is an integer and thus $\varepsilon_{n+1}=0$.
\end{proof}

\begin{example}\label{ex:1129}
    The pseudo-greedy expansion of $11/29$ is given as follows:
    \begin{center}
        \begin{tabular}{c||c|c|c|c|c|c|c}
            $n$ & $1$ & $2$ & $3$ & $4$ & $5$ & $6$ & $\cdots$ \rule[0pt]{0pt}{0pt} \\ \hline
            $x_n$ & $\dfrac{11}{29}$ & $\dfrac{15}{116}$ & $\dfrac{19}{1044}$ & $\dfrac{5}
            {14616}$ & $\dfrac{1}{10684296}$ & $\dfrac{1}{114154191699912}$ & $\cdots$ \rule[-10pt]{0pt}{25pt}\\ \hline
            $a_n$ & $4$ & $9$ & $56$ & $2924$ & $10684297$ & $114154191699913$ & $\cdots$  \rule[0pt]{0pt}{0pt} \\ \hline
            $\varepsilon_n$ & $-\dfrac{4}{11}$ & $-\dfrac{4}{15}$ & $-\dfrac{1}{19}$ & $\dfrac{1}{5}$ & $0$ & $0$ & $\cdots$ \rule[-10pt]{0pt}{25pt}
        \end{tabular}
    \end{center}
    We see that $(a_5,a_6,\dots)$ is the pseudo-greedy expansion of $1/10684296$, which is given as in \Cref{example:sylvester}.
    In other words, we have
    $\varepsilon_5=\varepsilon_6=\cdots=0$.
\end{example}

We conjecture that for any positive rational number $r$, the gap sequence of the pseudo-greedy expansion satisfies $\varepsilon_n=0$ for $n\gg 0$, at least when $\varepsilon_n\to 0$ (\Cref{conj}).
We confirmed by a computer that this conjecture is true for $r=p/q$ with $0<p\leq q\leq 10^5$.
Note that this conjecture resembles the termination problem of the odd greedy expansion.

Let us take a closer look at the pseudo-greedy expansion of positive rational numbers.

\begin{lemma}\label{lem:rational_PG}
    Let $r=p/q$ be a positive rational number, and $(a_n)_{n=1}^\infty$, $(x_n)_{n=1}^\infty$, $(\varepsilon_n)_{n=1}^\infty$ be its pseudo-greedy expansion, the remainder sequence, and the gap sequence.
    \begin{enumerate}
        \item Let $d_n=qa_1a_2\cdots a_{n-1}$.
        Then we can write
        \[
        x_n=\dfrac{c_n}{d_n},\quad \varepsilon_n = \dfrac{e_n}{c_n},
        \]
        where $c_n$ is a positive integer and $e_n$ is an integer.
        \item The pair $(c_n,d_n)$ determines $(e_n,a_n)$ by
        \[
        \begin{cases}
            e_n\equiv d_n\pmod{c_n},\quad -\dfrac{c_n}{2}\leq e_n<\dfrac{c_n}{2},\\
            a_n = \dfrac{d_n-e_n}{c_n}+1.
        \end{cases}
        \]
        The tuple $(c_n,d_n,e_n,a_n)$ determines $(c_{n+1},d_{n+1})$ by
        \[
        \begin{cases}
            c_{n+1}=c_n-e_n,\\
            d_{n+1}=d_na_n.
        \end{cases}
        \]
        \item We have the following asymptotic estimates:
        $$
            c_n=O\left(\dfrac{a_1a_2\cdots a_{n-1}}{a_n}\right),\quad
            c_n=O(1.5^n).
        $$
    \end{enumerate}
\end{lemma}

\begin{proof}
    \begin{enumerate}
        \item By definition, we have
    \[
    x_n = \dfrac{p}{q}-\sum_{k=1}^{n-1}\dfrac{1}{a_k}\in \dfrac{1}{qa_1\cdots a_{n-1}}\mathbb{Z}.
    \]
    Therefore, we can write $x_n=c_n/d_n$ for some positive integer $c_n$.
    The formula $\varepsilon_n=x_n^{-1}+1-a_n$ implies that $\varepsilon_n\in (1/c_n)\mathbb{Z}$, so we can write $\varepsilon_n = e_n/c_n$ for some integer $e_n$.
    \item The definition of $\varepsilon_n$ can be reformulated as $\varepsilon_n=x_n^{-1}-\lfloor x_n^{-1}\rceil$.
    Therefore, $\varepsilon_n$ is characterized by
    \[
    \varepsilon_n\equiv x_n^{-1} \pmod 1,\quad -\dfrac{1}{2}\leq \varepsilon_n < \dfrac{1}{2}.
    \]
    Since $\varepsilon_n=e_n/c_n$, this yields the characterization of $e_n$.
    The formula for $a_n$ follows from $a_n = x_n^{-1}  - \varepsilon_n + 1$.
    The equality $d_{n+1}=d_na_n$ is immediate from the definition.
    Finally, the formula for $c_{n+1}$ follows from
    $$
    x_{n+1}=x_n-\dfrac{1}{a_n}=\dfrac{c_na_n-d_n}{d_na_n} = \dfrac{c_n-e_n}{d_{n+1}}.
    $$
    \item The first estimate follows $d_n=qa_1a_2\cdots a_{n-1}$ and
    \[a_n = \lfloor x_n^{-1}+1\rceil\approx x_n^{-1}=\dfrac{d_n}{c_n}.\]
    The second one follows from $c_{n+1}= c_n-e_n\leq 1.5c_n$.\qedhere
    \end{enumerate}
\end{proof}

Now we prove that \Cref{ques:erdos} by Erd\H{o}s-Graham is equivalent to our \Cref{conj}:

\begin{theorem}\label{thm:equiv}
    \Cref{conj} is true if and only if \Cref{ques:erdos} has an affirmative answer.
\end{theorem}

\begin{proof}
    First, suppose that \Cref{conj} is true.
    Let $(a_n)_{n=1}^\infty$ be a sequence of positive integers satisfying
    $$
        \lim_{n\to \infty}\dfrac{a_n^2}{a_{n+1}}=1\quad\text{and}\quad
        \sum_{n=1}^\infty \dfrac{1}{a_n}\in \mathbb{Q}.
    $$
    By \Cref{cor:main}, we may assume that $(a_n)_{n=1}^\infty$ is the pseudo-greedy expansion of some positive rational number $r$, and that its gap sequence satisfies $\lim_{n\to \infty}\varepsilon_n=0$.
    By our assumption that \Cref{conj} is true, we have $\varepsilon_n=0$ for $n\gg 0$.
    The formula
    \begin{align}\label{eq:pseudo-sylvester}
        a_{n+1}=\dfrac{1}{1-\varepsilon_n}a_n^2-a_n+(1-\varepsilon_{n+1})
    \end{align}
    given in \Cref{lem:PG_sylvester} shows that $a_{n+1}=a_n^2-a_n+1$ for $n\gg 0$.

    Conversely, suppose that \Cref{ques:erdos} has an affirmative answer.
    Let $r$ be a positive rational number and $(a_n)_{n=1}^\infty$ be its pseudo-greedy expansion.
    We define the quantities $c_n$ and $e_n$ as in \Cref{lem:rational_PG}.
    Assume that the gap sequence satisfies $\lim_{n\to \infty}\varepsilon_n=0$.
    Then, the formula \eqref{eq:pseudo-sylvester} shows that
    $$
        \lim_{n\to \infty}\dfrac{a_n^2}{a_{n+1}}=1.
    $$
    Therefore, by our assumption that \Cref{ques:erdos} has an affirmative answer, we have $a_{n+1}=a_n^2-a_n+1$ for $n\gg 0$.
    Comparing this with the formula \eqref{eq:pseudo-sylvester}, we obtain
    \[
    \dfrac{\varepsilon_na_n^2}{1-\varepsilon_n}= \varepsilon_{n+1} = o(1).
    \]
    In particular, we have $\varepsilon_n=o(a_n^{-2})$.
    Combining this with the estimate $c_n=O(1.5^n)$ given in \Cref{lem:rational_PG} (3), we obtain
    \(e_n=c_n\varepsilon_n=o(1)\).
    Since $e_n$ is an integer, we see that $e_n=0$ holds for $n\gg 0$.
    This shows that \Cref{conj} is true.
\end{proof}

\begin{remark}\label{rem:heuristic}
Using \Cref{lem:rational_PG}, we can provide a heuristic argument showing that \Cref{conj} is likely to be correct even without assuming $\lim_{n\to \infty}\varepsilon_n=0$.
Let $r$ be a positive rational number, and define $c_n$, $d_n$, and $e_n$ as in \Cref{lem:rational_PG}.
By \Cref{lem:rational_PG} (2), we have $c_{n+1}=c_n-e_n$ and hence
\[
\dfrac{c_{n+1}}{c_n} = 1-\dfrac{e_n}{c_n},\quad -\dfrac{1}{2}\leq \dfrac{e_n}{c_n}< \dfrac{1}{2}.
\]
Thus the behavior of $(c_n)_{n=1}^\infty$ can be modeled by the multiplicative random walk
\(c_{n+1} = t_nc_n\),
where $t_n$ is chosen uniformly randomly from $[1/2,3/2)$.
Since we have
\[
\mathbb{E}[\log t_n] = \int_{1/2}^{3/2} \log t\ dt = \dfrac{3}{2}\log 3 - \log 2 - 1 = -0.0452287\cdots < 0,
\]
$c_n$ tends to shrink exponentially on average, so it is natural to expect that $e_n=0$ holds for some $n$.
\end{remark}

\begin{remark}
    When actually computing the values of \( \varepsilon_n \) for a given rational number \( r \), directly performing the calculation will cause the values of \( a_n \) and \( d_n \) to grow explosively large.
    We can use modular arithmetic to avoid this problem.
    When the values of \( c_1, \dots, c_n \) and \( e_1, \dots, e_{n-1} \) are known, one can use the recurrence relation from \Cref{lem:rational_PG} (2) to inductively compute, for \( k = 1, 2, \dots, n-1 \), the residue classes
    \[
        a_k \pmod{c_{k+1}c_{k+2}\cdots c_n}\quad\text{and}\quad d_{k+1} \pmod{c_{k+1}c_{k+2}\cdots c_n}.
    \]
    In particular, we can compute \( d_n \pmod{c_n} \), and this value can then be used to determine \( e_n \) and \( c_{n+1} \).
    The following is a pseudocode for computing the gap sequence $(\varepsilon_n)_{n=1}^\infty$ of the pseudo-greedy expansion of a positive rational number $r=p/q$.
    \begin{algorithm}[h]
    \caption{Pseudocode for computing $(\varepsilon_n)_{n=1}^\infty$ for $r=p/q$}
    \begin{algorithmic}[1]
    \State {$c_1\gets p$}
    \ForAll {$n \gets 1$ to $n_{\mathrm{max}}$}
        \State $d_1 \gets q$
        \ForAll {$k\gets 1$ to $n-1$}
            \State $a_k \gets ((d_k - e_k)/c_k) + 1 \pmod{c_{k+1}\cdots c_n}$
            \State $d_{k+1}\gets d_ka_k  \pmod{c_{k+1}\cdots c_n}$
        \EndFor
        \State $e_n \gets d_n - c_n\lfloor d_n/c_n\rceil$
        \State $c_{n+1} \gets c_n - e_n$
        \State $\varepsilon_n \gets e_n/c_n$
        \State Print $\varepsilon_n$
    \EndFor
    \end{algorithmic}
    \end{algorithm}
\end{remark}

Finally, we reinterpret the partial results on \Cref{ques:erdos} due to Erdős-Straus \cite{Erdos_Straus_63} and Badea \cite{Badea} within our framework.
In terms of the pseudo-greedy expansion, their results can be regarded as the following special cases of \Cref{conj}:

\begin{proposition}\label{prop:Badea}
    Let $r$ be a positive rational number, and $(\varepsilon_n)_{n=1}^\infty$ be the gap sequence of the pseudo-greedy expansion of $r$.
    Suppose that one of the following conditions is satisfied:
    \begin{enumerate}
        \item $\liminf_{n\to \infty}\varepsilon_n\prod_{k=1}^{n-1}(1-\varepsilon_k)\geq 0.$
        \item $\varepsilon_n\geq 0$ holds for $n\gg 0$.
    \end{enumerate}
    Then, we have $\varepsilon_n=0$ for $n\gg 0$.
\end{proposition}

\begin{proof}
    We define the quantities $c_n$ and $e_n$ as in \Cref{lem:rational_PG}.
    \begin{enumerate}
        \item By \Cref{lem:rational_PG} (2), we have $c_{n+1}=c_n-e_n$ and hence
            \[
                e_n=\varepsilon_nc_n=\varepsilon_nc_1\prod_{k=1}^{n-1}\left(1-\dfrac{e_k}{c_k}\right)=c_1\cdot\varepsilon_n\prod_{k=1}^{n-1}(1-\varepsilon_k).
            \]
            Our assumption shows that $\liminf_{n\to \infty}e_n\geq 0$.
            Since $e_n$ is an integer, we conclude that $e_n\geq 0$ holds for $n\gg 0$.
            Therefore, this case is reduced to (2).

        \item By \Cref{lem:rational_PG} (2), we have $c_{n+1} = c_n-e_n$ and hence $(c_n)_{n=1}^\infty$ is eventually non-increasing.
            Since $c_n$ is a positive integer, $(c_n)_{n=1}^\infty$ is eventually constant and thus $e_n=0$ holds for $n\gg 0$.\qedhere
    \end{enumerate}
\end{proof}

\begin{corollary}\label{cor:Badea}
    Let $(a_n)_{n=1}^\infty$ be a sequence of positive integers satisfying
    $$
        \lim_{n\to \infty}\dfrac{a_n^2}{a_{n+1}}=1\quad\text{and}\quad
        \sum_{n=1}^\infty \dfrac{1}{a_n}\in \mathbb{Q}.
    $$
    Suppose that one of the following conditions is satisfied:
    \begin{enumerate}
        \item (Erd\H{o}s-Straus \cite{Erdos_Straus_63}) $\liminf_{n\to \infty}\dfrac{a_1a_2\cdots a_{n-1}}{a_n}\left(1-\dfrac{a_n^2}{a_{n+1}}\right)\geq 0$.
        \item (Badea \cite{Badea}) $a_{n+1}\geq a_n^2-a_n+1$ holds for $n\gg 0$.
    \end{enumerate}
    Then, $a_{n+1}=a_n^2-a_n+1$ holds for $n\gg 0$.
\end{corollary}

\begin{proof}
    By \Cref{cor:main}, we may assume that $(a_n)_{n=1}^\infty$ is the pseudo-greedy expansion of some positive rational number $r$, and that its gap sequence satisfies $\lim_{n\to \infty}\varepsilon_n=0$.
    By \Cref{lem:PG_sylvester}, it suffices to shows that $\varepsilon_n=0$ holds for $n\gg 0$.
    \begin{enumerate}
    \item By \Cref{lem:PG_sylvester}, we have
    \(a_{n+1}=(a_n^2/(1-\varepsilon_n))(1+O(a_n^{-1}))\),
    and hence we can write
    \begin{align}\label{eq:ratio_vs_epsilon}
    \dfrac{a_n^2}{a_{n+1}} = (1-\varepsilon_n)(1+\beta_n),
    \end{align}
    where $\beta_n=O(a_n^{-1})$.
    We can write $a_1a_2\cdots a_{n-1}/a_n$ as
    \begin{align}\label{eq:product}
    \dfrac{a_1a_2\cdots a_{n-1}}{a_n}=\dfrac{1}{a_1}\prod_{k=1}^{n-1}\dfrac{a_k^2}{a_{k+1}}=\dfrac{1}{a_1}\prod_{k=1}^{n-1}(1-\varepsilon_k) \prod_{k=1}^{n-1}(1+\beta_k).
    \end{align}
    The infinite product $\prod_{k=1}^\infty(1+\beta_k)$ converges since $\beta_k=O(a_k^{-1})$.
    In particular, the quantity
    \[
    B_n:=\dfrac{1}{a_1}\prod_{k=1}^n(1+\beta_k)>0
    \]
    is bounded.
    Combining \eqref{eq:ratio_vs_epsilon} with \eqref{eq:product}, we obtain
    \[
    \dfrac{a_1a_2\cdots a_{n-1}}{a_n}\left(1-\dfrac{a_n^2}{a_{n+1}}\right)=B_{n-1} \varepsilon_n\prod_{k=1}^{n-1}(1-\varepsilon_k) - B_{n-1}\beta_n\prod_{k=1}^{n}(1-\varepsilon_k).
    \]
    The second term of the right hand side is $o(1)$ because $\beta_n=O(a_n^{-1})$ and  $|\varepsilon_k|\leq (1/2)$.
    Thus, our assumption is equivalent to
    \(
    \liminf_{n\to \infty}\varepsilon_n\prod_{k=1}^{n-1}(1-\varepsilon_k)\geq 0
    \).
    By \Cref{prop:Badea} (1), we conclude that $\varepsilon_n=0$ for $n\gg 0$.
    \item By \Cref{lem:PG_sylvester}, our assumption can be reformulated as
    \[
    \dfrac{\varepsilon_n a_n^2}{1-\varepsilon_n}\geq \varepsilon_{n+1}\quad (n\gg 0).
    \]
    In particular, for sufficiently large $n$, we have
    \[
    \varepsilon_n<0\implies \varepsilon_n>\varepsilon_{n+1}.
    \]
    Suppose that we have $\varepsilon_n<0$ for infinitely many $n$.
    Then, the above implication shows that $\varepsilon_n$ is decreasing for $n\gg 0$.
    This contradicts the fact that $\varepsilon_n$ converges to $0$.
    Therefore, we have $\varepsilon_n\geq 0$ for $n\gg 0$.
    By \Cref{prop:Badea} (2), we conclude that $\varepsilon_n=0$ for $n\gg 0$.\qedhere
    \end{enumerate}
\end{proof}

\begin{remark}
    The aforementioned result of Erd\H{o}s-Straus solves \Cref{ques:erdos} under the condition
    \[
    \dfrac{a_n^2}{a_{n+1}}=1+o(n^{-1}).
    \]
    Indeed, if we write $(a_n^2/a_{n+1})=1-\gamma_n$ with $\gamma_n=o(n^{-1})$, then we have
    \[
    \dfrac{a_1a_2\cdots a_{n-1}}{a_n}\left(1-\dfrac{a_n^2}{a_{n+1}}\right)=
    \dfrac{1}{a_1}\left(1-\dfrac{a_n^2}{a_{n+1}}\right)\prod_{k=1}^{n-1}\dfrac{a_k^2}{a_{k+1}} = \dfrac{\gamma_n}{a_1}\prod_{k=1}^{n-1}(1-\gamma_k) =o(1),
    \]
    so we can apply \Cref{cor:Badea} (1).
\end{remark}

\begin{remark}\label{rem:irr}
    Suppose that \Cref{ques:erdos} has an affirmative answer, or equivalently, that \Cref{conj} is true.
    Recall the set $\mathcal{I}$ from \Cref{type2_countable}, and let $\alpha\in (1,\infty)\setminus \mathcal{I}$.
    By definition, there is a sequence of positive integers $(a_n)_{n=1}^\infty$ satisfying $a_n\approx \alpha^{2^n}$ whose reciprocal sum is rational.
    By our assumption, $(a_n)_{n=1}^\infty$ eventually follows the recurrence relation
    \[
    a_{n+1}=a_n^2-a_n+1.
    \]
    In other words, there is some integer $N\geq 0$ and a positive integer $m$ such that \(a_{N+n} = s_n(m)\), where $s_n(m)$ is the sequence defined in \Cref{example:sylvester}.
    In particular, $\alpha$ is given by
    \[
    \alpha = \lim_{n\to \infty} a_{N+n}^{2^{-(N+n)}} = \lim_{n\to \infty} s_n(m)^{2^{-(N+n)}}=c(m)^{2^{-N}},
    \]
    where $c(m)$ is the constant defined in \Cref{example:sylvester}.
    Therefore, assuming that \Cref{ques:erdos} has an affirmative answer, we can conclude that
    \[
        (1,\infty)\setminus\mathcal{I}=\{c(m)^{2^{-N}}\mid N\geq 0,\ m>0\}.
    \]
    Since $c(m)$ is transcendental by the result of Dubickas \cite{Dubickas}, this implies that $\overline{\mathbb{Q}}\cap(1,\infty)\subset \mathcal{I}$.
    In particular, an affirmative answer to \Cref{ques:erdos} will imply that $2^{2^n}$ is a Type 2 irrationality sequence.
\end{remark}

\printbibliography

\end{document}